\documentclass[11pt]{amsart}
\usepackage{amsmath,amscd} 
\usepackage{latexsym,amsmath,amssymb,mathrsfs,setspace,enumerate}
\usepackage[T1]{fontenc}
\usepackage[utf8]{inputenc} 
\usepackage{amssymb}
\usepackage{tikz}
\usepackage{tikz-cd}

\usepackage[left=3cm, right=3cm, bottom=3cm,top=3cm]{geometry}
\newtheorem{thm}{Theorem}[section]

\newtheorem{propo}[thm]{Proposition}

\newtheorem{corol}[thm]{Corollary}
\newtheorem{defn}[thm]{Definition}

\newtheorem{remark}[thm]{Remark}

\begin{document}
\title{ On the Embedding of $ \Gamma $- semigroup Amalgam}
\date{}

\author{ Smisha M A $^1$ and  P. G. Romeo$^2$}
\address{$^1$ Assistant Professor, Department of Mathematics, St. Michael's College, Cherthala, Kerala, India} 
\address{$^2$ Professor, Department of Mathematics, Cochin University of Science and Technology, Kochi, Kerala, INDIA.}
\email{$^2 smishaashok@gmail.com , ^2 romeo_-parackal@yahoo.com\,,   $}

\begin{abstract}
$\Gamma$ semigroup is introduced as a generalization of semigroups by M. K Sen and Saha. In this paper we describe amalgam of two $ \Gamma $- semigroups and discuss the embeddability of this amalgam. Further we obtained a necessary condition for the embeddability of completely $ \alpha $-regular $ \Gamma $-semigroup amalgam.
\end{abstract}

\maketitle

A non-empty set together with a binary operation which is associative is known a semigroup. A semigroup 
is a more general algebraic structure than groups, obviously all groups are semigroups. Similar to that of free groups we have free semigroup of a generating set. For a detailed description of these concepts see \cite{HOW}.  
$ \Gamma $-semigroups is a generalization of semigroups analogous to the concept of $ \Gamma $-ring introduced by Nobusawa \cite{NOB}, \cite{BAR}.

\section{Preliminaries}
In the following first we briefly recall amalgams of groups, semigroups and free product of semigroup amalgams, then we 
recall $\Gamma $- semigroups, the amalgam of $ \Gamma $- semigroups and their properties as described by M K Sen and Saha  \cite{SEN}.

A group amalgam consists of three groups  $ G , K $ and  $ H $, called core of the amalgam and  embeddings $ m :H\rightarrow G $ and $ n :H\rightarrow K $. The amalgam is said to be embeddable if we can find another group $T$ which contains $ G \cup K $ containing  $H$. If $ G \cup K = H $ in $ T $, then  amalgam is called strongly embeddable.  Schreierr (1927) in \cite{SCHR} and  \cite{NEN} has shown that group amalgam is always embeddable. But all semigroup amalgams are not always embeddable.
 
\begin{defn} (cf.\cite{HOW} )
A semigroup amalgam  $\mathcal{K} = [\{ S_{i}; i \in I \} : U : \{ \theta_{i}; i \in I \}]$ consists of a semigroup $U$, the core of the amalgam, a family of mutually disjoint semigroups $\{ S_{i} : i \in I \}$ and a family of monomorphisms $\theta_{i} : U  \rightarrow S_{i}$\, $i \in I$.
\end{defn}

We say that the amalgam $\mathcal{K}$ is embedded in another semigroup $T$ if there exists
monomorphism $\mu : U \to T$ and for each $i \in I$ a monomorphism $\mu_{i} : S_{i} \to T$ with
\begin{enumerate}
\item  $\theta_{i}\mu_{i}=\mu$ for each $i \in I$.
\item $S_{i}\mu_{i} \cap S_{j}\mu_{j} = U \mu$ for all $i, j \in I$ such that $i \neq j.$
\end{enumerate}
The amalgam is called  weakly embeddable if it satisfies condition (1) only and is strongly embeddable
if  both conditions are satisfied. We say that the embedding of the amalgam is possible whenever such a semigroup $T$ exists.

A free product is an operation that takes two algebraic objects and construct a new object in the same category.
For a family of semigroups $\{ S_{i}:i \in I \}$ , the collection of all finite strings of the form $(a_{1},a_{2},\cdots , a_{n})$ for $a_{i} \in \{ \cup S_{i} ; i \in I$ \} is a semigroup with operation juxtaposition and is known as free product of semigroups denoted by $\Pi ^{*}S_{i}$.
The free product $\Pi ^{*}_{U}S_{i}$ of the amalgam $\mathcal{A}= [U;S_{i},\phi_{i}]$ is defined as the quotient semigroup of the ordinary free product $\Pi ^{*}S_{i}$ in which for each $i$ and $j$ in $I$ the image $\phi_{i}(u)$  of the image $u$ of $U$ in $S_{i}$ is identified its image $\phi_{j}(u)$ in $S_{j}$.
Howie \cite{HOW} showed that there always exist a mapping $ \mu_{i} $ from each $S_{i}$ of the amalgam to the amalgamated free product with the following properties: 
\begin{itemize}
\item each $\mu_{i}$ is one-one .
\item $ \mu_{i}(S_{i}) \cap \mu_{j}(S_{j}) \subseteq \mu(U) $ for all $i,j$ in $I$ such that $i \neq j$.
\end{itemize}
If the above two conditions hold, then we say that the amalgam $\mathcal{A} $ is \textit{naturally embedded in its free product} and based on this \cite{HOW} [Theorem 8.2.4] proved that the semigroup amalgam is embeddable in a semigroup if and only if it is naturally embedded in its free product.

\begin{defn} (cf.\cite{SEN})
Let $ S,\Gamma $ be two non-empty sets. Then $S$ is called a $\Gamma $-semigroup if there exist a mapping from $S \times \Gamma \times S$ to $S$ which maps $(a,\gamma,b)\rightarrow a \gamma b$ satisfying the condition $(a\gamma)b \mu c=a \gamma(b \mu c)$ for all $a,b,c \in S$ and $\gamma,\mu \in \Gamma$.
\end{defn}

Let $S$ be a $\Gamma$-semigroup, $A$ and $B$ are two subsets of $S$ , then the set  
$\{a\gamma b : a \in A , b \in B \, \gamma \in \Gamma\}$ is denoted by $A \Gamma B$. A nonempty subset $A\subset S$ of a $\Gamma$-semigroup is called  a  $\Gamma$-subsemigroup of $S$ if 
$A\Gamma A \subset A$. 
For $S,T$ be two $\Gamma$-semigroups with left identities $e$ and $e'$, 
a mapping $f:S\rightarrow T$ satisfying $f(e)=e'$ and $f(a\gamma b)=f(a) \gamma f(b)$ for each $a,b \in S$ and 
$\gamma \in \Gamma$ is called  a $\Gamma$-semigroup homomorphism.

Next we generalize amalgams of semigroups to amalgams of $\Gamma$-semigroups in terms of $\Gamma$-mappings. 

\begin{defn}
Let $U_{1}, S_{1}, S_{2}$ be mutually disjoint non-empty sets  and $\Gamma_{0}, \Gamma_{1}, \Gamma_{2}$ such that  $U$ is a $\Gamma_{0}$-semigroup called core of the amalgam, $S_{1}$ is a $\Gamma_{1}$-semigroup, $S_{2}$ is $\Gamma_{2}$-semigroup. Then the collection $\textsc{A}=[(U,\Gamma_{0}),(S_{1},\Gamma_{1}),(S_{2},\Gamma_{2}),f_{1},f_{2}]$ constitutes a $\Gamma$-semigroup amalgam where $f_{i}=(f_{i}^{'},f_{i}^{''})$ are $\Gamma$-monomorphisms such that $f_{i}^{'}:U \to S_{i}$ and $f_{i}^{''}:\Gamma_{0} \to \Gamma_{i}$ for $i=\{1,2 \}$.
\end{defn}

A  $\Gamma$-semigroup amalgam $\textsc{A}=[(U,\Gamma_{0}),(S_{1},\Gamma_{1}),(S_{2}, \Gamma_{2}),f_{1},f_{2}]$ is said to be embeddable in another $\Gamma$-semigroup $(T,\Gamma_{T})$  if there exist mappings ($\Gamma$-monomorphisms)$g=(g',g''):(U,\Gamma_{0}) \to (T,\Gamma_{T})$ and  $g_{i}=(g_{i}^{'},g_{i}^{''}) : (S_{i},\Gamma_{i}) \to (T,\Gamma_{T})$ where $g_{i}: S_{i} \to T$ and $g_{i}^{''}:  \Gamma_{i} to \Gamma_{T}$ for $i=1,2$  and $g': U \to T$ , $g'' : \Gamma_{0} \to \Gamma_{T}$ such that the following holds true:
\begin{itemize}
\item[\textbf{[1]}] : the diagram \\
\begin{tikzcd}
(U,\Gamma_{0}) \arrow[r,"f_{1}"] \arrow [d,"f_{2}"] \arrow[dr,"f"] & (S_{1},\Gamma_{1}) \arrow[d, "g_{1}"] \\
 (S_{2},\Gamma_{2}) \arrow[r,"f_{2}"] & (T,\Gamma_{T})
\end{tikzcd} commutes and 
\item[\textbf{[2]}] : $g_{1}((S_{1},\Gamma_{1})) \cap g_{2}((S_{2},\Gamma_{2})) = g((U, \Gamma_{0}))$
\end{itemize}

The second condition can be restated as: 
\begin{itemize}
\item[•] whenever $g_{1}(s_{1}\gamma_{1}s_{1}^{'})=g_{2}(s_{2}\gamma_{2}s_{2}^{'})$ (for $s_{1},s_{1}^{'} \in S_{1}$, $s_{2},s_{2}^{'} \in S_{2}$, $\gamma_{1} \in \Gamma_{1}$ and $\gamma_{2} \in \Gamma_{2}$) there exists $u \in U$ and $\gamma_{0} \in \Gamma_{0}$ such that $f_{i}^{'}(u)=s_{i}$ and $f_{i}^{''}(\gamma_{0})=\gamma_{i}$ for $i=1,2$.
\end{itemize}

If the amalgam satisfies condition $[1]$ then it is called weakly $\Gamma$-embeddable and if the amalgam satisfies  both conditions, then it is called strongly $\Gamma$-embeddable.

\section{Free product of $\Gamma$-semigroups}

Free product of an idexed family of pairwise disjoint semigroups was introduced by J M Howie \cite{HOW} and the notion of free $\Gamma$-semigroup was introduced by M K Sen \cite{SEN}.  In the following we define  free product of a family of mutually disjoint  $\Gamma$-semigroups $\{S_{i}:i \in I\}$ and a family $\{ \Gamma_{i}: i \in T\}$ which are also mutually disjoint. 
\begin{defn}
Let $\{ \Gamma_{i}: i \in T\}$ be a family of disjoint non-empty sets, $\{S_{i}:i \in I\}$ be a family of mutually disjoint  $\Gamma$-semigroups and  $x \in \cup \{S_{i} : i \in I\} $. Then there is a unique $l$ in $I$ such that $x \in S_{l} $ and $l$ is called pointer of $x$, denoted  by $l=\zeta(x)$. Similarly for  an element $\gamma \in \cup \{ \Gamma_{i}: i \in I\}$,  there is a unique $m$ in $I$ such that $\gamma \in \Gamma_{m}$ and $m$ is called pointer of $\gamma$.
\end{defn}

Consider $ \Delta $ as the collection of all elements of the form $ (x_{1}, \gamma_{1}, x_{2}, \gamma_{2}, \cdots , \gamma_{m-1}, x_{m} ) $ where $ m \geq 1$ is an integer, $x_{n} \in \cup \{S_{i} : i \in I\}$ for $n=1,2,3,\cdots, m$ , $\gamma_{k} \in \cup \{ \Gamma_{i}: i \in I\} $ for  $k=1,2,3,\cdots, m-1$ and $\zeta(x_{n}) \neq \zeta(x_{n+1})\neq \zeta(\gamma_{n})$ for $n=1,2,3,\cdots, m-1$ of all finite strings.
Define the product of two elements in $\Delta$ as follows: \\
$\textbf{a}=(a_{1}, \gamma_{1}, a_{2}, \gamma_{2}, \cdots , \gamma_{m-1}, a_{m} )$ and $ \textbf{b}=(b_{1}, \alpha_{1}, b_{2}, \alpha_{2}, \cdots , \alpha_{n-1}, b_{n} ) $ are elements in $ \Delta $ and  $\gamma \in \Gamma $, the 
product of $ \textbf{a} $ and $ \textbf{b} $ in $ \Delta $ is : 
\[
 \textbf{a $ \gamma $ b} = 
\begin{cases}
(a_{1}, \gamma_{1},  \cdots , \gamma_{m-1}, a_{m} \gamma b_{1}, \alpha_{1}, b_{2}, \alpha_{2}, \cdots , \alpha_{n-1}, b_{n} ) & \text{if } \zeta(a_{m} = \zeta(\gamma)= \zeta(b_{1}) \\
(a_{1}, \gamma_{1},  \cdots , \gamma_{m-1}, a_{m}, \gamma, b_{1}, \alpha_{1}, b_{2}, \alpha_{2}, \cdots , \alpha_{n-1}, b_{n} ) & \text{otherwise} 
\end{cases}
\]
\begin{thm}
Let $ \Delta =\{ (x_{1}, \gamma_{1}, x_{2}, \gamma_{2}, \cdots , \gamma_{m-1}, x_{m} ) \}$ where $ m \geq 1$ is an integer, $x_{n} \in \cup \{S_{i} : i \in I\}$ for $n=1,2,3,\cdots, m$ , $\gamma_{k} \in \cup \{ \Gamma_{i}: i \in I\} $ for  $k=1,2,3,\cdots, m-1$. Then  $ \Delta $ is a $ \Gamma $- semigroup
\end{thm}
\begin{proof}
For associativity, we have to consider the following cases. Let  
\begin{eqnarray*}
\textbf{a} &= &(a_{1}, \gamma_{1}, a_{2}, \gamma_{2}, \cdots , \gamma_{m-1}, a_{m} )\\
 \textbf{b} &=& (b_{1}, \alpha_{1}, b_{2}, \alpha_{2}, \cdots , \alpha_{n-1}, b_{n} ) \\
  \textbf{c} &= & (c_{1}, \beta_{1}, c_{2}, \beta_{2}, \cdots , \beta_{p-1}, c_{p} ) 
\end{eqnarray*} 
in $ \Delta $ and $\alpha , \beta  $ in $ \cup \{ \Gamma_{i}: i \in I\}$. \\[4 pt]
When 
$\zeta (a_{m}) = \zeta (\alpha)= \zeta(b_{1})$ and $\zeta (b_{n}) = \zeta (\beta)= \zeta(c_{1})$ \\ 
$\textbf{a $ \alpha $  b } = (a_{1}, \gamma_{1}, a_{2}, \gamma_{2}, \cdots , \gamma_{m-1}, a_{m} \alpha  b_{1}, \alpha_{1}, b_{2}, \alpha_{2}, \cdots , \alpha_{n-1}, b_{n} )$  \\
$\textbf{b $ \beta $ c}= b_{1}, \alpha_{1}, b_{2}, \alpha_{2}, \cdots , \alpha_{n-1}, b_{n} \beta c_{1}, \beta_{1}, c_{2}, \beta_{2}, \cdots , \beta_{p-1}, c_{p}) $ 
\begin{eqnarray*}
\textbf{(a $ \alpha $ b )  $ \beta $ c } & = & (a_{1}, \gamma_{1},  \cdots , \gamma_{m-1}, a_{m}  \alpha  b_{1}, \alpha_{1},  \cdots , \alpha_{n-1}, b_{n} \beta c_{1}, \beta_{1},  \cdots , \beta_{p-1}, c_{p})\\
& = & \textbf{a $ \alpha $ ( b  $ \beta $  c) }  
\end{eqnarray*}
In a similar way one can see the associativity follows in other situations also. Hence $ \Delta  $ is a $ \Gamma $- semigroup.
\end{proof}

\subsection*{Free product of disjoint $ \Gamma $- semigroups }
Here we are considering the case of two $ \Gamma $- semigroups over the same $ \Gamma $ and the case of a family of disjoint $ \Gamma $- semigroups over the same $ \Gamma $ is similar in nature.\\
 Consider non-empty sets $S_{1},S_{2},\Gamma$ and let $ S_{1} $ and $ S_{2} $ are $ \Gamma $- semigroups over the same $ \Gamma $. As earlier, the set of all elements of the form $(a_{1}, \gamma_{1}, a_{2}, \gamma_{2}, \cdots , \gamma_{m-1}, a_{m} )$ where $a_{1}, a_{2}, \cdots ,  a_{m} \in S_{1} \cup S_{2}$ and $ \gamma_{1},  \gamma_{2}, \cdots , \gamma_{m-1} \in \Gamma $ is a $ \Gamma $- semigroup and is called free-$ \Gamma $-semigroup product over the alphabet $ S_{1} \cup S_{2} $ relative to $ \Gamma $. We denote it by $ S_{i}^{*}\Gamma $. The empty word is the word which has a no letters. Operation on $ S_{i}^{*}\Gamma $ is defined as same as that we defined earlier.\\
 For \begin{eqnarray*}
 \textbf{a} & = & (a_{1}, \gamma_{1}, a_{2}, \gamma_{2}, \cdots , \gamma_{m-1}, a_{m} )\\
 \textbf{b} & = & (b_{1}, \alpha_{1}, b_{2}, \alpha_{2}, \cdots , \alpha_{n-1}, b_{n} ) 
\end{eqnarray*} 
 in $ S_{i}^{*} \Gamma $ and $ \gamma \in \Gamma $
\[
 \textbf{a $ \gamma $ b} = 
\begin{cases}
(a_{1}, \gamma_{1}, a_{2}, \gamma_{2}, \cdots , \gamma_{m-1}, a_{m} \gamma b_{1}, \alpha_{1}, b_{2}, \alpha_{2}, \cdots , \alpha_{n-1}, b_{n} ) & \text{if } \zeta(a_{m} =  \zeta(b_{1}) \\
(a_{1}, \gamma_{1}, a_{2}, \gamma_{2}, \cdots , \gamma_{m-1}, a_{m}, \gamma, b_{1}, \alpha_{1}, b_{2}, \alpha_{2}, \cdots , \alpha_{n-1}, b_{n} ) & \text{if } \zeta(a_{m} \neq  \zeta(b_{1}) 
\end{cases}
\]
It is straight forward to check associativity for the various cases arise, and so  we can see that 
$ S_{i}^{*}\Gamma $ is a $ \Gamma $-semigroup.

\begin{remark}
There are elements of $ S_{i}^{*}\Gamma $ which are strings $ (a_{i}) $ of length one where $a_{i} \in S_{1} \cup S_{2}$. In fact $ S_{i}^{*}\Gamma $ is generated by strings of length one, since $$ (a_{1}, \gamma_{1}, a_{2}, \gamma_{2}, \cdots , \gamma_{m-1}, a_{m} )=(a_{1}), (\gamma_{1}), (a_{2}), (\gamma_{2}), \cdots , (\gamma_{m-1}), (a_{m} ) $$  for every ($ a_{1}, \gamma_{1}, a_{2}, \gamma_{2}, \cdots , \gamma_{m-1}, a_{m} )$ in $ S_{i}^{*}\Gamma $. So we exclude 
the brackets in writing a string of length one.
\end{remark}

A vital property of free products is the following.
\begin{propo} \label{propo 1}
Let $ \Delta = S_{i}^{*}\Gamma $ be the free-$ \Gamma $-product of two $ \Gamma $-semigroups over the same $ \Gamma $. Then for each $ i \in \{1,2 \} $ there exists a $ \Gamma $-monomorphism $ \theta_{i}: (S_{i},\Gamma) \to \Delta $ given by, $\theta_{i}(a \gamma b)=(a \gamma b)  $ associating the elements $a,b \in S_{1} \cup S_{2}$ and $\gamma \in \Gamma $ with the one letter word $(a \gamma b)$. If $T$ is a $ \Gamma $-semigroup over the same $ \Gamma $ for which there is a $ \Gamma $-morphism $\psi_{i}:S_{i} \to T $ for each $i=1,2$ then there is a unique morphism $\lambda: \Delta \to T$ with the property that the diagram  
 $$ \begin{tikzcd}
S_{i} \arrow[r,"\theta_{i}"] \arrow [d,"\psi_{i}"]  & \Delta \arrow[dl, "\lambda"] \\
 T & 
\end{tikzcd} $$ 
commutes for every $i=1,2$.
 \end{propo} 

\begin{proof}
 Since each element $a \gamma b$ in $S_{1} \cup S_{2}$ is also a word $(a \gamma b)$ in $ \Gamma $, we can consider 
 $S_{1} \cup S_{2}$ as a $ \Gamma $-subsemigroup of $ \Delta $. Given $T$ and $ \Gamma $-morohisms $ \psi_{i} $ are given $(i=1,2)$, define $ \lambda : \Delta  \to T $ by, 
 $$ \lambda((a_{1}, \gamma_{1}, a_{2}, \gamma_{2}, \cdots , \gamma_{m-1}, a_{m} ))=\psi_{\zeta(a_{1})}(a_{1}), \gamma_{1},\psi_{\zeta(a_{2})}(a_{2}) ,\gamma_{2}, \cdots, \gamma_{m-1}, \psi_{\zeta(a_{m})}(a_{m}) $$.

The expression on the right side is the product of  elements of $T$, hence $ \lambda $ maps $ \Delta $ into $T$. 
Further $ \lambda $ is a $ \Gamma $-morphism, for  
 \begin{eqnarray*}
 \textbf{a} &= &(a_{1}, \gamma_{1}, a_{2}, \gamma_{2}, \cdots , \gamma_{m-1}, a_{m} )\\
 \textbf{b} &=& (b_{1}, \alpha_{1}, b_{2}, \alpha_{2}, \cdots , \alpha_{n-1}, b_{n} ) \\ 
 \end{eqnarray*} elements of $ \Delta $ and $ \gamma \in \Gamma $, then if $\zeta(a_{m}) \neq \zeta(b_{1})$, 
 
 \begin{eqnarray*}
 \lambda(a \gamma b) & = & \lambda((a_{1}, \gamma_{1}, a_{2}, \gamma_{2}, \cdots , \gamma_{m-1}, a_{m} , \gamma , b_{1}, \alpha_{1}, b_{2}, \alpha_{2}, \cdots , \alpha_{n-1}, b_{n} )) \\
 & = & \psi_{\zeta(a_{1})}(a_{1}) ,\gamma_{1}, \cdots,\gamma , \psi_{\zeta(b_{1})}(b_{1}), \alpha_{1}, \cdots \psi_{\zeta(b_{n})}(b_{n}) \\
 & = & [ \psi_{\zeta(a_{1})}(a_{1}) ,\gamma_{1}, \cdots,\gamma_{m-1}, a_{m} )]\gamma [ \psi_{\zeta(b_{1})}(b_{1}), \alpha_{1}, \cdots ,\alpha_{n-1},\psi_{\zeta(b_{n})}(b_{n}) \\
 & = & \lambda(a) \gamma \lambda(b)
 \end{eqnarray*}
 The above diagram commutes since 
 \begin{eqnarray*}
 \lambda \theta_{i}(a \gamma b) & =  & \lambda (\theta_{i}(a \gamma b)) \\
 & = & \lambda ((a \gamma b)) \\
 & = & \psi_{\zeta(a)}(a) \gamma \psi_{\zeta(b)}(b) \\
 & = & \psi_{i}(a \gamma b) 
 \end{eqnarray*}
 hence $\lambda \theta_{i}= \psi_{i}$.
 
 The uniqueness of the $ \lambda $ follows from the way that $ \Delta $ is created by words of length one. If $ \lambda $ is to make the diagram commute then it is compulsory to have $ \psi_{i}(a \gamma b)= \lambda ((a \gamma b)) $ for every $i=1,2$. Then, if $ \lambda $ is to be a $ \Gamma $- homomorphism we must have 
 \begin{eqnarray*}
 \lambda(a_{1}, \gamma_{1}, a_{2}, \gamma_{2}, \cdots , \gamma_{m-1}, a_{m}) & = & \lambda(a_{1}), \gamma_{1},\lambda( a_{2}), \gamma_{2}, \cdots , \gamma_{m-1}, \lambda(a_{m}) \\
 & = & \psi_{\zeta(a_{1})}(a_{1})\gamma_{1}\psi_{\zeta(a_{2})}(a_{2})\gamma_{2} \cdots \gamma_{m-1} \psi_{\zeta(a_{m})}(a_{m})
 \end{eqnarray*}
 for every $ a_{1}, \gamma_{1}, a_{2}, \gamma_{2}, \cdots , \gamma_{m-1}, a_{m}$ in $ \Delta $. \\
 ie., $ \Delta $ must be exactly what we defined it.
 \end{proof}

 \begin{propo}
Let  $S_{1}$ and $ S_{2} $ be two $ \Gamma $-semigroups , $\Delta $ be the free product of $S_{1}$ and $ S_{2} $ and let $H$ be a $ \Gamma $-semigroup such that 
   \begin{itemize}
 \item there exists a $ \Gamma $-monomorphism $f_{i}:S_{i} \to H$ for $i=1,2$. 
 \item if $T$ is a $ \Gamma $-semigroup  and if there exist $ \Gamma $-monomorphism $g_{i}: S_{i} \to T$ for $i=1,2$ then there exists a unique  $ \Gamma $-monomorphism $\delta  : H \to T$ such that the diagram
 $$ \begin{tikzcd} 
S_{i} \arrow[r,"f_{i}"] \arrow [d,"g_{i}"]  & H \arrow[dl, "\delta"] \\
 T & 
\end{tikzcd}  $$ \\
commutes for every $i=1,2$. 
 \end{itemize}
 then $H$ is isomorphic to $ \Delta  $.
 
 \end{propo}
\begin{proof}

From the property of $ \Delta $  in Proposition \ref{propo 1}, when $ T = \Delta $ and $ \psi_{i}= \phi_{i} $ for $ i=1,2 $ there is a unique $ \Gamma $-morphism (identity map $ 1_{\Delta} $) $ \Delta \to \Delta $ making the diagram 
 $$ \begin{tikzcd} 
S_{i} \arrow[r] \arrow [d,"\theta_{i}"]  & \Delta \arrow[dl] \\
 \Delta & 
\end{tikzcd}  $$ \\
commutes for every $ i=1,2 $.

By uniqueness  $ 1_{\Delta} $ is the only $ \Gamma $-morphism from $ \Delta $ into $ \Delta $ having the property. Similarly the identity map $ 1_{H} $ is the only $ \Gamma $-morphism from $H$ into $H$ with the property that the diagram 
 $$ \begin{tikzcd} 
S_{i} \arrow[r,"f_{i}"] \arrow [d,"f_{i}"]  & H \arrow[dl] \\
 H & 
\end{tikzcd}  $$ \\
commutes for every $i=1,2$. 

Now apply Proposition \ref{propo 1} with $T=H$ and $ \psi_{i}=f_{i} $ for $i=1,2$ then we have  a morphism $ \lambda $ from $ \Delta $  to $H$ such that  
\begin{center}
\begin{equation}  \label{d1}
\begin{tikzcd} 
S_{i} \arrow[r,"\theta_{i}"] \arrow [d,"f_{i}"]  & \Delta \arrow[dl, "\lambda"] \\
 H & 
\end{tikzcd}
\end{equation}
\end{center}  
commutes for every $i=1,2$. 

Then by the assumed property of $H$ with $T= \Delta $ and $g_{i} = \theta_{i} $ we obtain a $ \Gamma $-morphism $ \delta : H \to \Delta $ such that the diagram 
\begin{center}
\begin{equation}  \label{d2}
\begin{tikzcd} 
S_{i} \arrow[r,"f_{i}"] \arrow [d,"\theta_{i}"]  & H \arrow[dl, "\delta"] \\
 \Delta & 
\end{tikzcd}
\end{equation}
\end{center}  
commutes for every $i=1,2$.

It follows that 
\begin{center}
$ \delta \lambda \theta_{i} = \delta f_{i} = \theta_{i} $\\
$ \lambda \delta f_{i} = \lambda \theta_{i} = f_{i} $
\end{center}
that is, if we track together, the dragrams \ref{d1} and \ref{d2} in both of the possible ways we obtain commutative diagrams 
\begin{center}
\begin{equation*}  
\begin{tikzcd} 
S_{i} \arrow[r,"\theta_{i}"] \arrow [d,"\theta_{i}"]  & \Delta \arrow[dl, "\delta \lambda"] \\
 \Delta & 
\end{tikzcd} 
\begin{tikzcd} 
S_{i} \arrow[r,"f_{i}"] \arrow [d,"f_{i}"]  & H \arrow[dl, "\lambda \delta"] \\
 H & 
\end{tikzcd}
\end{equation*}
for $i=1,2$
\end{center} 
by the uniqueness , we get, 
\begin{center}
$ \delta \lambda = 1_{\Delta} $ and $ \lambda \delta = 1_{H}$
\end{center}
thus $ \delta $ and $ \lambda $ are mutually inverse $ \Gamma $-isomorphisms and $ H \simeq \Delta $ as required.
 \end{proof}
We denote by \textbf{$ \Gamma $ -Sgrp} the category of $ \Gamma $-semigroups which has the $ \Gamma $-semigroups as objects and the $ \Gamma $-homomorphisms of $ \Gamma $-semigroups as arrows.

Then by the above two propositions we can conclude that $ \Delta $ is the unique coproduct in the sense of category theory.

\section{Congruences and Free $\Gamma $ product}

In this section we discuss certain congruences on $ \Gamma $-semigroups.
\begin{defn}(cf.\cite{SETH})
Let S be a $ \Gamma $-semigroup. An equivalence relation $ \rho $ on $S$ is called congruence if $ x \rho y$ implies that $(x \gamma z) \rho (y \gamma z)$ and  $(z \gamma  x \rho (z \gamma y)$ for all $x,y,z \in S$ and $\gamma \in \Gamma$.
\end{defn}

Let $\rho $ be a congruence relation on $(S,\Gamma )$. Then $S / \rho = \{ \rho (x) : x \in S  \}$ is the set of all equivalence classes of elements of $S$ with respect to $ \rho $.
\begin{thm} (cf.\cite{CRAM})
Let $ \rho $ be a congruence relation on $ \Gamma $-semigroup $(S,\Gamma )$. Then $S / \rho $ is a $ \Gamma $-semigroup.
\end{thm}
\begin{thm}(cf.\cite{HED})
Let $(\phi , g): (S_{1},\Gamma_{1}) \to (S_{2},\Gamma_{2})$ be a homomorphism. Define the relation $ \rho_{(\phi , g)} $ on $ (S_{1},\Gamma_{1}) $ as follows:
$$ x \rho_{(\phi , g)} y \Longleftrightarrow \phi(x) = \phi(y). $$ Then  $ \rho_{(\phi , g)} $ is a congruence on $ (S_{1},\Gamma_{1}) $.
\end{thm}
\begin{thm} (cf.\cite{CRAM}) \label{thm3}
Let $S$ and $T$ be $ \Gamma $-semigroups under the same $ \Gamma $ and $ \phi: S \to T $ be a $ \Gamma $-homomorphism. Then there is a $ \Gamma $-homomorphism $\psi : s / ker\phi \to T $ such that $im \phi = im \psi $ and the diagram 
$$ \begin{tikzcd} 
S \arrow[r,"\phi"] \arrow [d,"(ker \phi )^{\#}"]  & T \arrow[dl, "\psi"] \\
 S/ker \phi & 
\end{tikzcd} $$  commutes where $ (ker \phi )^{\#} $ is the natural mapping from $S$ into $S / ker\phi$ defined by $ (ker \phi )^{\#} (x) = x ker \phi $ for all $x \in S $.
\end{thm}
\begin{corol}
$ S / ker \phi \cong im \phi .$
\end{corol} 
\subsection*{Free $ \Gamma $-product of amalgam }
Free $ \Gamma $- product   $\textsc{A}=[(U,\Gamma_{0}),(S_{1},\Gamma_{1}),(S_{2},\Gamma_{2}),f_{1},f_{2}]$ of the amalgam writted as $ \Delta_{U} $ is defined as the quotient semigroup of the ordinary free $ \Gamma $- product $ \Delta $ in which for each $i = 1,2 $; the image $ f_{1}(u\gamma_{0}u') $ of an element $ u\gamma_{0}u' $ of $ (U,\Gamma_{0}) $ in $ (S_{1},\Gamma_{1}) $ is identified with its image $ f_{2}(u\gamma_{0}u') $ in $ (S_{2},\Gamma_{2}) $.

More precisely, denote $ \theta_{i} $ the natural $ \Gamma $-monomorphism  from $S_{i}$ to $\Delta = S_{i}^{*}\Gamma $,  then we define $ \Delta_{U}= S_{i}^{*}\Gamma /\rho $ where $ \rho $ is the congruence on $ \Delta $ generated by the subset 
\begin{equation} \label{eqn}
\textbf{R}= \{ (\theta_{1} f_{1}(u\gamma_{0}u'),\theta_{2} f_{2}(u\gamma_{0}u')): u\gamma_{0}u' \in U \}
\end{equation}   
 of $ \Delta \times \Delta$.

It is clear that for each $i=1,2$ there is a morphism $ \mu_{i}=\rho^{\#}\theta_{i} $ from $ S_{i} \to \Delta_{U} $. It is also clear from the definition of $ \rho $ that we have a commutative diagram 
$$
\begin{tikzcd} 
U \arrow[r,"f_{1}"] \arrow [d,"f_{2}"]  & S_{1} \arrow[d, "\mu_{1}"] \\
S_{2} \arrow[r,"\mu_{1}"] & \Delta_{U}
\end{tikzcd} 
$$
So there exists a  $ \Gamma $-morphism $\mu : U \to \Delta_{U}$ such that $\mu = \mu_{1}f1=\mu_{2}f_{2}$.\\
Since $\mu(U) \leqslant \mu_{i}(S_{i})$, we necessarily have $\mu(U) \leqslant \mu_{1}(S_{1}) \cap \mu_{2}(S_{2})  $. \\
Hence the $ \Gamma $-amalgam is embedded in its free $ \Gamma $ product if and only if 
\begin{itemize}
\item each $ \mu_{i} $ is one-one  
\item $ \mu_{1}(S_{1}) \cap \mu_{2}(S_{2}) \leqslant \mu(U) $
\end{itemize}
If the above two conditions hold, we say that the amalgam is naturally $ \Gamma $-embedded in its free $ \Gamma $ product.
\begin{thm}
The $ \Gamma $-semigroup amalgam $\textsc{A}=[(U,\Gamma_{0}),(S_{1},\Gamma_{1}),(S_{2},\Gamma_{2}),f_{1},f_{2}]$ is embeddable in a $ \Gamma $-semigroup if and only if it is naturally $ \Gamma $-embedded in its free $ \Gamma $ product.
\end{thm} 

Before moving to the proof, we have to make use of the following result.
\begin{propo} \label{propo3}
If $\textsc{A}$ is a $ \Gamma $ semigroup amalgam, the the free $ \Gamma $ product $ \Delta_{U} $ of the amalgam is the pushout of the diagram $\{ U \to S_{1} \}_{i=1,2}$. That is 
\begin{itemize}
\item[(1)]  There exists for each $i = 1,2 $ a $ \Gamma $-morphism $\mu_{i} : S_{i} \to \Delta_{U}$ such that the diagram $ \{ U \to S_{i} \to \Delta_{U} \}_{i=1,2}$ commutes. That is,  $ \mu_{1}f_{1}=\mu_{2}f_{2} $. 
\item[(2)]  If $V$ is a $ \Gamma $-semigroup for which morphisms $ g_{i}:S_{i} \to V $ exist such that  $ g_{1}f_{1}=g_{2}f_{2} $, then there exists a unique $ \Gamma $- morphism $ \delta :  \Delta_{U} \to V $ such that the diagram 
$$
\begin{tikzcd} 
S_{i} \arrow[r,"\mu_{i}"] \arrow [d,"g_{i}"]  & \Delta_{U} \arrow[dl, "\delta"] \\
V & 
\end{tikzcd} 
$$ commutes for each $i=1,2$.
\end{itemize}
\end{propo}
\begin{proof}
We have already observed property $[1]$. \\
To see  property $[2]$, notice that by Proposition \ref{propo 1}  there is a unique morphism $ \lambda : \Delta \to v $ such that 
$$
\begin{tikzcd} 
S_{i} \arrow[r,"\theta_{i}"] \arrow [d,"g_{i}"]  & \Delta \arrow[dl, "\lambda"] \\
V & 
\end{tikzcd} 
$$  is a commutative diagram for each $i=1,2$.

Now, for all $ u,u' \in U $ and $\gamma_{0} \in \Gamma $
\begin{eqnarray*}
\lambda \theta_{1}f_{1}(u\gamma_{0}u') &=& g_{1}f_{1}(u\gamma_{0}u') \\
& = & g_{2}f_{2}(u\gamma_{0}u') \\
& = & \lambda \theta_{2}f_{2}(u\gamma_{0}u')
\end{eqnarray*}
Hence $ ( \theta_{1}f_{1}(u\gamma_{0}u'), \theta_{2}f_{2}(u\gamma_{0}u')) \in \lambda \circ \lambda^{-1}$.
So from equation (\ref{eqn}) $\textbf{R} \subseteq \lambda \circ \lambda^{-1} $, 
hence, since $ \lambda \circ \lambda^{-1} $ is a congruence  
$$  \rho = \textbf{$R^{\#}$} \subseteq \lambda \circ \lambda^{-1}  $$
By Theorem \ref{thm3} it follows that the morphism $ \lambda : \Delta \to V $ factors through $ \Delta_{U}=S/\rho $,

that is, there is a unique morphism $\delta : \Delta_{U} \to V$ such that the diagram 
$$
\begin{tikzcd} 
\Delta \arrow[r,"\rho^{\#}"] \arrow [d,"\lambda"]  & \Delta_{U} \arrow[dl, "\delta"] \\
V & 
\end{tikzcd} 
 commutes. $$
By the definition of $g_{i}$ and from the commutativity of the above two diagrams,  for  $i=1,2$ 
$$ g_{i}=\lambda \theta_{i}= \delta \rho^{\#}\theta_{i}=\delta \mu_{i} $$
hence the diagram commutes.  
\end{proof}
Now we are going back to the proof of the theorem.
\begin{proof}
One way is obvious and suppose that $\textsc{A}$ is embeddable in a $ \Gamma $-semigroup $T$, so that there exists $ \Gamma $-monomorphisms $g_{i}=(g_{i}^{'},g_{i}^{''}): (S_{i},\Gamma_{i}) \to (T,\Gamma_{T})$   for $i=1,2 $ and $g=(g',g''): (U,\Gamma_{0}) \to (T,\Gamma_{T})$ such that $f_{1}g_{1}=f_{1}g_{2}=g$ and such that $g_{1}(S_{1}) \cap g_{2}(S_{2})= g(\lambda)$.

By Proposition \ref{propo3} there exists a unique $ \Gamma $-morphism $ \delta : \Delta_{U} \to T $ such that 
$$
\begin{tikzcd} 
S_{i} \arrow[r,"\mu_{i}"] \arrow [d,"g_{i}"]  & \Delta_{U} \arrow[dl, "\delta"] \\
T & 
\end{tikzcd} $$
is a commutative diagram for each $i=1,2$.

For $s\gamma s^{'}, s_{1}\gamma_{1}s_{1}^{'} \in (S_{1},\Gamma_{1})$ , if 
$\mu_{1}(s\gamma s^{'})= \mu_{1}(s_{1}\gamma_{1}s_{1}^{'}) $ then \\
Suppose
\begin{eqnarray*}
\delta( \mu_{1}(s\gamma s^{'})) &=& \delta(\mu_{1}(s_{1}\gamma_{1}s_{1}^{'})) \,\text{ so}\\
  g_{1}(s\gamma s^{'}) &=& g_{1}(s_{1}\gamma_{1}s_{1}^{'}) \\
  s\gamma s^{'} &=& s_{1}\gamma_{1}s_{1}^{'} 
\end{eqnarray*} 
\begin{center}
thus $ \mu_{1} $ is a $ \Gamma $-monomorphism.
\end{center}
Similarly we get $ \mu_{2} $ is a $ \Gamma $-monomorphism.

Suppose  $x \in \mu_{1}S_{1} \cap \mu_{2}S_{2}$, without loss of generality, assume $x=  \mu_{1}(s_{1}\gamma_{1}s_{1}^{'})=\mu_{2}(s_{2}\gamma_{2}s_{2}^{'})$ where $s_{1},s_{1}^{'} \in S_{1}$, $s_{2},s_{2}^{'} \in S_{2}$, $\gamma_{1} \in \Gamma_{1}$ and $\gamma_{2} \in \Gamma_{2}$. Then $\delta(x)=\delta(\mu_{1}(s_{1}\gamma_{1}s_{1}^{'}))= g_{1}(s_{1}\gamma_{1}s_{1}^{'}) \in g_{1}(S_{1})$. Similarly, $\delta(x) \in g_{2}(S_{2})$. Thus $\delta(x) \in g_{1}(S_{1}) \cap g_{2}(S_{2})=g(U)$ and so there exist $u=u\gamma_{0}u' \in U $ such that $\delta(x)=g(U)$.

That is 
\begin{eqnarray*}
g_{1}(s_{1}\gamma_{1}s_{1}^{'}) &=& \delta \mu_{1}(s_{1}\gamma_{1}s_{1}^{'})\\
&=& \delta(x)\\
&=& g(u)\\
&=& g_{1}f_{1}(u\gamma_{0}u')
\end{eqnarray*}
similarly
\begin{eqnarray*}
g_{2}(s_{2}\gamma_{2}s_{2}^{'}) &=& \delta \mu_{2}(s_{2}\gamma_{2}s_{2}^{'})\\
&=& \delta(x)\\
&=& g(u)\\
&=& g_{2}f_{2}(u\gamma_{0}u')
\end{eqnarray*}
Since $g_{i}$ is $ \Gamma $-monomorphism it follows that $f_{1}(u\gamma_{0}u')=s_{1}\gamma_{1}s_{1}^{'} $ and $f_{2}(u\gamma_{0}u')=s_{2}\gamma_{2}s_{2}^{'} $ and so $x= \mu_{1}(s_{1}\gamma_{1}s_{1}^{'})= \mu_{1}(f_{1}(u\gamma_{0}u'))$ and $x= \mu_{2}(s_{2}\gamma_{2}s_{2}^{'})=\mu_{2}f_{2}(u\gamma_{0}u') \in \mu(U)$.\\
Thus $\textsc{A}$ is naturally $ \Gamma $-embedded  in $ \Delta_{U} $ as required.
\end{proof}

\section{On Amalgam of Completely $ \alpha $-Regular  $ \Gamma $-Semigroups}
In the following we define completely $ \alpha $-regular  $ \Gamma $-semigroups and provide  necessary condition to the embeddability of their amalgam.
\begin{defn}
An element $a \in S$ is called  $ \alpha $-regular if there exist elements $x \in S$ and  $\alpha  \in \Gamma$ such that $a=a \alpha x \alpha a$. If every element of a $\Gamma$-semigroup is $ \alpha $-regular  in $S$ then $S$ is called $ \alpha $-regular $\Gamma$-semigroup.
\end{defn}
\begin{defn}
For an $ \alpha $-regular $\Gamma$-semigroup $S$, an element $b \in S$ is called $ \alpha $- inverse of an element $a \in S$ if $a=a \alpha b \alpha a$ and $b=b \alpha a \alpha b$ for all $a,b \in S$ and $\alpha  \in \Gamma$.
\end{defn}
\begin{defn}
An $ \alpha -$regular semigroup $S$ is called $ \Gamma $-inverse $\Gamma$-semigroup if every element in $S$ has a unique $ \alpha $-inverse in $S$.
\end{defn}
\begin{defn} An element $a$ is completely $ \alpha $-regular  if there exist an element $x$ in $S$ such that $a=a \alpha x \alpha a$ and $a \alpha x = x \alpha a$ for all $a, x \in S$ and $\alpha \in \Gamma $. If all the elements in a $\Gamma$-semigroup $S$ is completely $ \alpha $-regular then $S$ is called completely $ \alpha $-regular $\Gamma$-semigroup.
\end{defn}

\begin{thm}
A $\Gamma$-semigroup amalgam $(U; S_{1}, S_{2},\phi_{1},\phi_{2})$, in which $S_1$ and
$S_2$ are completely $ \alpha $-regular $\Gamma$-semigroups over the same $ \Gamma $ is embeddable only if $U$ is also completely $ \alpha $-regular.
\end{thm}
\begin{proof}
Consider an amalgam $A = (U; S_1; S_2; \phi_1; \phi_2)$ in which $S_1$ and $S_2$
are both completely $ \alpha $-regular $\Gamma$-semigroups. Suppose $ \textsc{A}$ is embeddable, say in a $\Gamma$-semigroup $T$ and consider an element $u \in U$.\\
Let us denote $\phi_{1}(u ) = s_1 $ and $\phi_{2}(u) = s_2 $.\\
Since $S_1$ and $S_2$ are completely $ \alpha $-regular, there exist inverses $s_{1}^{-1} \in S_1$
and $s_{2}^{-1} \in S_2$ of $s_1  $ and $s_2 $ respectively. Let  $\psi_{i}: S_{i} \to T, i \in {1,2} $ be the embedding $ \Gamma $-monomorphisms, then $$ \psi_{1}(s_1) =  \psi_{1}\phi_{1}(u) =  \psi_{2}\phi_{2}(u) =\psi_{2}(s_2)$$. 
We can calculate in $T$:
\begingroup
\allowdisplaybreaks
\begin{align*}
\psi_{2}(s_{2}^{-1}) &= \psi_{2}(s_{2}^{-1} \gamma s_{2} \gamma s_{2}^{-1} ) \\
& = \psi_{2}(s_{2}^{-1}) \gamma \psi_{2}(s_{2}) \gamma \psi_{2}(s_{2}^{-1}) \\
&= \psi_{2}(s_{2}^{-1}) \gamma \psi_{1}(s_{1}) \gamma \psi_{2}(s_{2}^{-1}) \\
&= \psi_{2}(s_{2}^{-1}) \gamma \psi_{1}(s_{1} \gamma s_{1}^{-1} \gamma s_{1}) \gamma \psi_{2}(s_{2}^{-1}) \\
&= \psi_{2}(s_{2}^{-1}) \gamma \psi_{1}(s_{1}) \gamma \psi_{1}( s_{1}^{-1}) \gamma \psi_{1} (s_{1}) \gamma \psi_{2}(s_{2}^{-1}) \\
&= \psi_{2}(s_{2}^{-1}) \gamma \psi_{2}(s_{2}) \gamma \psi_{1}( s_{1}^{-1}) \gamma \psi_{1} (s_{1}) \gamma \psi_{2}(s_{2}^{-1}) \\
&= \psi_{2}(s_{2}^{-1} \gamma s_{2}) \gamma \psi_{1}( s_{1}^{-1}) \gamma \psi_{1} (s_{1}) \gamma \psi_{2}(s_{2}^{-1}) \\
&= \psi_{2}(s_{2}^{-1} \gamma s_{2}) \gamma \psi_{1}( s_{1}^{-1} \gamma s_{1}) \gamma \psi_{2}(s_{2}^{-1}) \\
&= \psi_{2}(s_{2}^{-1} \gamma s_{2}) \gamma \psi_{1}(s_{1}  \gamma s_{1}^{-1} ) \gamma \psi_{2}(s_{2}^{-1}) \\
&= \psi_{2}(s_{2}^{-1} \gamma s_{2}) \gamma \psi_{1}(s_{1} )  \gamma \psi_{1}(s_{1}^{-1} ) \gamma \psi_{2}(s_{2}^{-1}) \\
&= \psi_{2}(s_{2}^{-1} \gamma s_{2}) \gamma \psi_{2}(s_{2} )  \gamma \psi_{1}(s_{1}^{-1} ) \gamma \psi_{2}(s_{2}^{-1}) \\
&= \psi_{2}(s_{2}  \gamma s_{2}^{-1}) \gamma \psi_{2}(s_{2} )  \gamma \psi_{1}(s_{1}^{-1} ) \gamma \psi_{2}(s_{2}^{-1}) \\
&= \psi_{2}(s_{2}  \gamma s_{2}^{-1} \gamma s_{2} )  \gamma \psi_{1}(s_{1}^{-1} ) \gamma \psi_{2}(s_{2}^{-1}) \\
&= \psi_{2}(s_{2} )  \gamma \psi_{1}(s_{1}^{-1} ) \gamma \psi_{2}(s_{2}^{-1}) \\
&= \psi_{1}(s_{1} )  \gamma \psi_{1}(s_{1}^{-1} ) \gamma \psi_{2}(s_{2}^{-1}) \\
&= \psi_{1}(s_{1}  \gamma s_{1}^{-1} ) \gamma \psi_{2}(s_{2}^{-1}) \\
&= \psi_{1}(s_{1}^{-1}  \gamma s_{1} ) \gamma \psi_{2}(s_{2}^{-1}) \\
&= \psi_{1}(s_{1}^{-1})  \gamma \psi_{1}(s_{1} ) \gamma \psi_{2}(s_{2}^{-1}) \\
&= \psi_{1}(s_{1}^{-1})  \gamma \psi_{2}(s_{2} ) \gamma \psi_{2}(s_{2}^{-1}) \\
&= \psi_{1}(s_{1}^{-1})  \gamma \psi_{2}(s_{2}  \gamma s_{2}^{-1}) \\
&= \psi_{1}(s_{1}^{-1})  \gamma \psi_{2}(s_{2}^{-1}  \gamma s_{2}) \\
&= \psi_{1}(s_{1}^{-1})  \gamma \psi_{2}(s_{2}^{-1})  \gamma \psi_{2}(s_{2}) \\
&= \psi_{1}(s_{1}^{-1})  \gamma \psi_{2}(s_{2}^{-1})  \gamma \psi_{1}(s_{1}) \\
&= \psi_{1}(s_{1}^{-1})  \gamma \psi_{2}(s_{2}^{-1})  \gamma \psi_{1}(s_{1} \gamma s_{1}^{-1} \gamma s_{1}) \\
&= \psi_{1}(s_{1}^{-1})  \gamma \psi_{2}(s_{2}^{-1})  \gamma \psi_{1}(s_{1}) \gamma \psi_{1}(s_{1}^{-1}) \gamma \psi_{1}(s_{1}) \\
&= \psi_{1}(s_{1}^{-1})  \gamma \psi_{2}(s_{2}^{-1})  \gamma \psi_{2}(s_{2}) \gamma \psi_{1}(s_{1}^{-1}) \gamma \psi_{1}(s_{1}) \\
&= \psi_{1}(s_{1}^{-1})  \gamma \psi_{2}(s_{2}^{-1}  \gamma s_{2}) \gamma \psi_{1}(s_{1}^{-1} \gamma s_{1}) \\
&= \psi_{1}(s_{1}^{-1})  \gamma \psi_{2}(s_{2}  \gamma s_{2}^{-1}) \gamma \psi_{1}(s_{1} \gamma s_{1}^{-1}) \\
&= \psi_{1}(s_{1}^{-1})  \gamma \psi_{2}(s_{2}  \gamma s_{2}^{-1}) \gamma \psi_{1}(s_{1}) \gamma \psi_{1}(s_{1}^{-1}) \\
&= \psi_{1}(s_{1}^{-1})  \gamma \psi_{2}(s_{2}  \gamma s_{2}^{-1}) \gamma \psi_{2}(s_{2}) \gamma \psi_{1}(s_{1}^{-1}) \\
&= \psi_{1}(s_{1}^{-1})  \gamma \psi_{2}(s_{2}  \gamma s_{2}^{-1} \gamma s_{2}) \gamma \psi_{1}(s_{1}^{-1}) \\
&= \psi_{1}(s_{1}^{-1})  \gamma \psi_{2}(s_{2} ) \gamma \psi_{1}(s_{1}^{-1}) \\
&= \psi_{1}(s_{1}^{-1})  \gamma \psi_{1}(s_{1} ) \gamma \psi_{1}(s_{1}^{-1}) \\
&= \psi_{1}(s_{1}^{-1}  \gamma s_{1}  \gamma s_{1}^{-1}) \\
&= \psi_{1}(s_{1}^{-1})
\end{align*}
\endgroup
Now, using condition (2) of embeddability, there exists $u' \in  U$ such that
$\phi_{i}(u') = s_{i}^{-1}$ for $i \in \{ 1, 2\}$.  Then, because $s_{1} \gamma s_{1}^{-1} \gamma s_{1} = s_{1}$ implies $\phi_{1}^{-1}(s_{1} \gamma s_{1}^{-1} \gamma s_{1}) =\phi_{1}^{-1}( s_{1})$
, we have $u \gamma u' \gamma u =u$ due to injectivity of $\phi_{1}$. Similarly we can conclude
that $u' \gamma u \gamma u' = u'$ and  $u \gamma u' = u' \gamma u$. Thus $U$ is completely $\Gamma$-regular.
\end{proof}

\end{document}